\documentclass[a4paper,12pt]{article}
\usepackage{amssymb,enumerate}
\usepackage{amsmath,amssymb}
\usepackage{amsfonts}
\usepackage{amsthm}
\newtheorem{theorem}{Theorem}
\newtheorem{df}{Definition}
\newtheorem{lem}{Lemma}
\newtheorem{prop}{Proposition}
\newtheorem{cor}{Corollary}

\title{New classes of Picard operators}

\author{{Pa\c sc G\u avru\c ta, Laura Manolescu} }
\date{}

\begin{document}

\maketitle

\begin{minipage}{120mm}
\small{\bf Abstract.} { An operator $T$ on metric space $(X,d)$ is called a Picard operator if $T$ has a unique fixed point $u$ in $X$ and for any $x\in X$, the sequence $\{T^nx\}_{n\in\mathbb{N}}$ converge to $u$.

In this paper, we give new results concerning the existence of Picard operators.
}\\

{\bf Keywords} {Picard operator fixed point}\\


\end{minipage}
\section{Introduction}

The fixed point theorems have various application in chemistry, biology,
computer sciences, differential equations, existence of invariant subspaces of
linear operators, Hyers-Ulam-Rassias stability and much more. Because of this, many scientists work on developing new fixed point theorems. See, for example, the book \cite{Rus3}. 
\\

Let $(X,d)$ be a metric space and $T:X\rightarrow X$ be a mapping.

$T$ is called a \textit{Picard operator} if $T$ has a unique fixed point $u$ in $X$ and for any $x\in X$, the sequence $\{T^nx\}_{n\in\mathbb{N}}$ converge to $u$ (\cite{Rus1}, \cite{Rus2}) .

\begin{df} $T$ is called contractive if it satisfies \[d(Tx, Ty)< d(x,y),\quad\textrm{for}~x,y\in X,~x\neq y.\]
\end{df}
The following proposition is a well known result.
\begin{prop}\label{firstprop}
Let $(X,d)$ be a complete metric space and $T:X\rightarrow X$ be contractive. If $(\forall)~x\in X$ the sequence $\{T^nx\}$ is Cauchy, then $T$ is Picard operator.
\end{prop}
\begin{proof}\cite{Meir}
Since $(X,d)$ is complete, the sequence $\{T^nx\}$ has a limit $u$. Since $T$ is continuous,$$Tu=T(\lim T^nx)=\lim T^{n+1}x=u.$$
Thus, $u$ is a fixed point of $T.$ If $v$ is such that $Tv=v,$ then $v=u.$ Contrary, $u\neq v$ and $$d(Tu, Tv)<d(u,v)\Longleftrightarrow d(u,v)<d(u,v),$$ impossible.
\end{proof}
The condition in Proposition \ref{firstprop} is not enough in general to ensure that existence of a fixed point. But, for contractions, the existence and the uniqueness of a fixed point are proved by the famous theorem of S. Banach.
\begin{df}
$T$ is called a contraction if there is $\lambda\in[0,1)$ such that 
$$d(Tx, Ty)\leq\lambda d(x,y),\quad\textrm{for}~x,y\in X.$$
\end{df}
\begin{theorem}(Banach \cite{Banach})
 Let $(X,d)$ be a complete metric space and\\ $T:X\rightarrow X$ be a contraction. Then $T$ is a Picard operator.
 \end{theorem}
 The Banach Theorem is an abstract formulation of Picard iterative process.
 In the following, we present some known generalizations of this theorem.
 \begin{df} We say that  $T$ is a Meir-Keeler contraction if given  $\varepsilon>0,$ there exists $\delta>0$ such that
 $$(\forall)~x,y\in X,~\varepsilon\leq d(x,y)<\varepsilon+\delta \Longrightarrow~d(Tx, Ty)<\varepsilon.$$
 \end{df}
\begin{theorem}(\cite{Meir}) Let $(X,d)$ be a complete metric space and $T$ be a Meir-Keeler contraction. Then $T$ is a Picard operator.
\end{theorem}
\begin{df}\label{dfmc} $T$ is said to be a \textit{CJMP contraction} (cf. \cite{Ciric}, \cite{Matkowski}, \cite{Jachymski}, \cite{Proinov}) if the following conditions holds
\begin{enumerate}[$(a)$]
    \item $T$ is contractive;
    \item (Matkowski-W\c egrzyk condition \cite{MatkowskiWe}) for every $\varepsilon>0,$ there exists $\delta=\delta(\varepsilon)>0$ such that
    $$(\forall)~x,y\in X, \varepsilon<d(x,y)<\varepsilon+\delta\Longrightarrow d(Tx,Ty)\leq\varepsilon.$$
\end{enumerate}
\end{df}
The class of CJMP-contractions contains the class of  Meir-Keeler contractions \cite{Ciric}.
\begin{theorem}(\cite{Ciric}, \cite{Matkowski})\label{thcjmp}
Let $(X,d)$ be a complete metric space and $T$ be a CJMP-contraction on $X$. Then $T$ is a Picard operator.
\end{theorem}
We will give a pedagogical proof of the above theorem. In the proof we will need the next lemma. 
\begin{lem} We suppose that $(X, d)$ is a metric space and $T:X\rightarrow X$
 a CJMP-contraction. Let $\varepsilon>0$ and $\delta=\delta(\varepsilon)>0$ as in Definition \ref{dfmc}. If $x,y,z\in X$ so that $d(x,y)<\delta$ and $d(y,z)\leq\varepsilon,$ then $$d(Tx,Tz)\leq\varepsilon.$$
\end{lem}
\begin{proof}
If $x=z,$ it is clear. We can suppose that $x\neq z.$ We have two cases:
\begin{enumerate}[$1)$]
    \item $d(x,z)\leq\varepsilon.$ Since $T$ is contractive, we have $d(Tx,Tz)<d(x,z)\leq\varepsilon.$ 
    \item $d(x,z)>\varepsilon.$ We have $\varepsilon<d(x,z)<d(x,y)+d(y,z)<\delta+\varepsilon,$\\ hence $d(Tx,Tz)\leq\varepsilon.$
\end{enumerate}
\end{proof}
\begin{proof}\textit{of Theorem \ref{thcjmp}.}  We take $x\in X$ and we denote $x_n=T^nx,$ $n\in\mathbb{N}.$\\
If $d(x_n, x_{n+1})=0$ for an $n\in\mathbb{N},$ then $x_n$ is a fixed point of $T$ and the proof is finished. If $d(x_n, x_{n+1})>0,$ $(\forall) n\in\mathbb{N},$ we give the proof in two steps.\\
\\
\textit{Step 1:} $\displaystyle\lim_{n\rightarrow \infty} d(x_n,x_{n+1})=0.$ We denote: $a_n=d(x_n, x_{n+1}),~n\in\mathbb{N}.$
Since $T$ is contractive, we have $a_n<a_{n-1},$ $n\in\mathbb{N}$. We denote $a:=\lim a_n.$ If $a>0,$ then there is $n_0(a)$ such that $$a<a_n<a+\delta(a),~n\geq n_0(a).$$
Since $(b),$ it follows $a_{n+1}\leq a,$ contradiction.\\
\\
\textit{Step 2:} $\{x_n\}$ is a Cauchy sequence. From \textit{Step 1}, we have that for $\varepsilon>0,$ $(\exists) n_1=n_1(\varepsilon)$ so that $$d(x_{n-1},x_n)<\gamma(\varepsilon):=min(\varepsilon,\delta(\varepsilon)),~n\geq n_1.$$
We use the induction to prove \begin{equation}\label{eqstar}
d(x_n,x_{n+p})\leq \varepsilon, p=1,2,\ldots
\end{equation}
For $p=1:$ $d(x_n, x_{n+1})<d(x_{n-1},x_n)<\varepsilon.$
If (\ref{eqstar}) is true, with the above Lemma, we have $$d(Tx_{n-1},Tx_{n+p})\leq\varepsilon,~n\geq n_1$$ and we apply Proposition \ref{firstprop}.
\end{proof}
Recently, in 2012, Wardowski introduced a new type of contraction mappings named $F-$ contractions (or Wardowski contractions) \cite{Wardowski2}. This new type of contractions were used by several researches in the field of fixed point theory to obtain new results. More general, in 2018, Wardowski also considered nonlinear $F-$contraction (or $(\varphi,F)$-contraction).
\begin{df}(Wardowski \cite{Wardowski})
Let be $F:(0,\infty)\rightarrow\mathbb{R}$ and $\varphi:(0,\infty)\rightarrow (0,\infty).$ $T$ is said to be a nonlinear $F-$contraction (or $(\varphi,F)$-contraction) if
\begin{equation}\label{eqW}
   (\forall)~x,y\in X, Tx\neq Ty\Longrightarrow \varphi(d(x,y))+
   F(d(Tx, Ty))\leq F(d(x,y)).  
\end{equation}
\end{df}
\begin{theorem} (\cite{Wardowski})
Let $(X,d)$ be a complete metric space and $T:X\rightarrow X$ be a $(\varphi,F)$-contraction. We suppose that 
\begin{enumerate}[$(i)$]
    \item $F$ is strictly increasing;
    \item $\displaystyle \lim_{t\rightarrow 0^+}F(t)=-\infty;$
    \item $\displaystyle \liminf_{s\rightarrow t^+}\varphi(s)>0,~\textrm{for all}~ t\geq 0.$
\end{enumerate}
Then $T$ is a Picard operator.
\end{theorem}

Important contributions to Wardowski contractions were given in \cite{Popescu}, \cite{Secelean1},\cite{Secelean2}, \cite{Suzuki}, \cite{Turinici} and \cite{Fulga}. See also the survey paper \cite{Karapinar}.\\
\\
We recall, also, a general theorem for fixed points.
\begin{theorem}(Ri \cite{Ri}) Let $(X, d)$ be a complete metric space and $T$ be contractive map in the following
sense: there is a function $\varphi:[0,\infty)\rightarrow [0,\infty)$ such that $\varphi(t)<t$ and $\displaystyle \limsup_{s\rightarrow t^+}\varphi(s)<t$ for all $t>0$ and $$d(Tx,Ty)\leq\varphi(d(x,y)),\quad (\forall)~x,y\in X.$$
Then $T$ has a unique fixed point $p$ in $X.$
\end{theorem}
%In this paper, we prove that nonlinear $F-$contractions $T$ are CJMP-contractions if $T$ is contractive and $F$ has finit limit at right at any point. As a consequence, we improve some known results regarding nonlinear $F-$contractions.
In this paper, we give two general theorems of existence and uniqueness for fixed point for applications on complete metric spaces.  Among other results, we generalize the above mention theorems of Wardowski \cite{Wardowski} and Ri\cite{Ri}. Also, we will improve some results of Gubran, Alfaqih, Imdad \cite{Gubran} and Proinov \cite{Proinov2}.
\section{The main results}

\begin{df}
Let $E,F$ be two real functions defined on $(0,\infty).$ We say that $(E, F)$ is a compatible pair of functions if the following conditions holds
\begin{enumerate}[$(C_1)$]
\item For $t,s\in(0,\infty),$ $t\leq s\Rightarrow$ $E(t)<F(s);$
\item Given $t>0$ and $(t_n)_{n\in\mathbb{N}}\subset(t,\infty)$ be a sequence with $\displaystyle\lim_{n\rightarrow\infty}t_n=t,$ then for any sequence $(s_n)_{n\in\mathbb{N}},$ $t<s_n<t_n,$ $n\in\mathbb{N}$ we have $$\limsup_{n\rightarrow\infty}(F(s_n)-E(t_n))>0.$$
\end{enumerate}
\end{df}
\begin{df}
We say that $T$ is an $(E,F)$-contraction if $(E,F)$ is a compatible pair of functions such that
\begin{equation}\label{(E,F)-contractions}
Tx\neq Ty\Rightarrow F(d(Tx,Ty))\leq E(d(x,y)).
\end{equation}
\end{df}

\begin{theorem}\label{ourthmlmg} Let $(X,d)$ be a complete metric space and $T:X\rightarrow X$ be a $(E,F)$-contraction. Then $T$ is a CJMP-contraction, hence a Picard operator.
\end{theorem}
\begin{proof}
First, we prove that $T$ is contractive. We suppose that $x\neq y$ and we prove that $$d(Tx,Ty)<d(x,y).$$ 
If $Tx=Ty$ this is clear. If $Tx\neq Ty,$ we suppose that $d(Tx,Ty)\geq d(x,y).$ By condition $(C_1)$ it follows $$F(d(Tx,Ty))> E(d(x,y)),$$ contradiction with (\ref{(E,F)-contractions}).\\

We prove that $T$ verifies the condition $(b)$ in Definition \ref{dfmc}. Contrary, there is $\varepsilon_0>0$ such that for any $\delta>0,$ there are $x_{\delta}, y_{\delta}\in X$ such that $$\varepsilon_0<d(x,y)<\varepsilon_0+\delta~\textrm{and}~d(Tx_{\delta}, Ty_{\delta})>\varepsilon_0.$$
We take $\delta=2^{-n},~n\in\mathbb{N}.$ Then there are two sequences $\{x_n\}_{n\in\mathbb{N}}, \{y_n\}_{n\in\mathbb{N}}\subset X$ such that
\begin{equation}\label{oureqthm1}
\varepsilon_0<d(x_n,y_n)<\varepsilon_0+2^{-n}
\end{equation}
and
\begin{equation}\label{oureqthm2}
d(Tx_n, Ty_n)>\varepsilon_0,~n\in\mathbb{N}.
\end{equation}
From these relations, with notations $$t_n=d(x_n,y_n),~s_n=d(Tx_n,Ty_n),~n\in\mathbb{N}.$$ 
We obtain that $$\{s_n\}_{n\in\mathbb{N}},\{t_n\}_{n\in\mathbb{N}}\subset(\varepsilon_0,\infty),$$
$\displaystyle\lim_{n\rightarrow\infty}t_n=\varepsilon_0$ and since $T$ is contractive, we have also $s_n<t_n,$ $n\in\mathbb{N}.$\\
By $(C_2),$ we have $\displaystyle\limsup (F(s_n)-E(s_n))>0.$
By (\ref{(E,F)-contractions}), we have $F(s_n)\leq E(t_n),~n\in\mathbb{N},$ hence $$\limsup_{n\rightarrow\infty}(F(s_n)-E(t_n))\leq 0,$$
contradiction.
\end{proof}
 Examining the proof of the above theorem, we observe that it takes place a more general result:
 \begin{theorem}\label{ourthmlmg2}
 Let $(X,d)$ be a complete metric space and $T:X\rightarrow X$ be a contractive mapping, which satisfies the relation (\ref{(E,F)-contractions}), where $E,F$ verifies condition $(C_2)$. Then $T$ is a CJMP-contraction, hence a Picard operator.
 \end{theorem}
\section{Applications}

In the following, we denote by $\mathcal{R}_+$ the set of all real valued functions defined on $(0,\infty)$, which have finite limit at right in any point.

For $F\in\mathcal{R}_+,$ we denote $$F(t+0):=\lim_{s\rightarrow t^+}F(s),~t>0.$$

We begin with a Lemma.
\begin{lem}\label{ourlem}
Let $(a_n)_{n\in\mathbb{N}},$ $(b_n)_{n\in\mathbb{N}}$ be two sequences of real numbers such that $(a_n)$ is convergent and $(b_n)$ is bounded. Then $$\limsup_{n\rightarrow\infty}(a_n+b_n)=\lim_{n\rightarrow\infty}a_n+\limsup_{n\rightarrow\infty} b_n.$$
\end{lem}
\begin{proof}
We denote by $L(x_n)$ the set of limit points for real sequence $(x_n).$ We prove that $$L(a_n+b_n)=a+L(b_n),$$ where $\displaystyle a:=\lim_{n\rightarrow\infty}a_n.$ Indeed,
$$c\in L(a_n+b_n)\Longleftrightarrow (\exists) (n_k)\subset\mathbb{N}:a_{n_k}+b_{n_k}\rightarrow c$$
which is equivalent with $$(\exists) (n_k)\subset\mathbb{N}:\lim b_{n_c}=c-a\Longleftrightarrow c-a\in L(b_n)\Longleftrightarrow c\in a+L(b_n).$$
Then $$\sup L(a_n+b_n)=a+\sup L(b_n).$$
\end{proof}

APPLICATION 1. In Theorem \ref{ourthmlmg}, we take the particular case $F(s)=s,$ $s\in(0,\infty)$ and $E$ be such that $E(t)<t,$ $t\in(0,\infty).$ Then $(C_1)$ is true:
$$t\leq s\Rightarrow E(t)<t\leq s=F(s).$$
Now, we consider $t>0$ and $(t_n)\subset (t,\infty)$ such that $$\lim_{n\rightarrow\infty} t_n=t.$$
If $(s_n)$ is a sequence such that $t<s_n<t_n,$ $n\in\mathbb{N}$ then condition from $(C_2)$ is equivalent with $$\limsup (s_n-E(t_n))>0\Longleftrightarrow t-\liminf E(t_n)>0,$$ by Lemma \ref{ourlem}.

Thus, we obtain the following result.
\begin{theorem}\label{thm8}
Let $(X,d)$ be a complete metric space and $T:X\rightarrow X.$ Let $E$ be a function $E:(0,\infty)\rightarrow (0,\infty)$ so that 
\begin{enumerate}[$\circ$]
    \item $E(t)<t$,\quad $t>0$
    \item $\displaystyle\liminf_{n\rightarrow\infty}E(t_n)<t,$ $t>0,$ for any $(t_n)\subset(t,\infty)$ with $\displaystyle\lim_{n\rightarrow\infty}t_n=t.$
    
    If $d(Tx,Ty)\leq E(d(x,y)),$ then $T$ is a Picard operator.
\end{enumerate}
\end{theorem}
Theorem \ref{thm8} improve the main result of \cite{Ri}.\\
\\
APPLICATION 2. We take $E(t)=F(t)-\varphi(t),~t>0,$ where\\ $\varphi:(0,\infty)\rightarrow(0,\infty).$ We take $F:(0,\infty)\rightarrow\mathbb{R}$ such that $F$ is nondecreasing:
$$t\leq s\Rightarrow F(t)\leq F(s).$$ Then $t\leq s\Rightarrow E(t)<F(t)\leq F(s).$\\
\\
We verify the condition $(C_2).$ Let be $t>0$ and $(t_n)\subset (t,\infty)$, with $\displaystyle\lim_{n\rightarrow\infty} t_n=t$ and a sequence $(s_n),$ $t<s_n<t_n,n\in\mathbb{N}.$
Condition $(C_2)$ is equivalent with $$\limsup(F(s_n)-F(t_n)+\varphi(t_n))>0$$
and from Lemma \ref{ourlem} this is equivalent with $$F(t+0)-F(t+0)+\limsup\varphi(t_n)>0.$$
\begin{theorem}
We suppose that $F:(0,\infty)\rightarrow\mathbb{R}$ is such that $F$ is nondecreasing and $\varphi:(0,\infty)\rightarrow (0,\infty)$ is such that,$$(iii')~\textrm{for all}~ t>0,  \displaystyle\limsup_{n\rightarrow\infty}\varphi(t_n)>0~\textrm{if}~ (t_n)\subset(t,\infty)~\textrm{and}~t_n\rightarrow t.$$ If $T:X\rightarrow X$ is so that $$Tx\neq Ty\Rightarrow \varphi(d(x,y))+F(d(Tx,Ty))\leq F(d(x,y)),$$ then $T$ is a Picard operator.
\end{theorem}
Using Theorem \ref{ourthmlmg2}, it takes place a more general result, which improve  Theorem 2.1 in \cite{Wardowski}, Theorem 4 in  \cite{Vujakovic} and generalize Corollary 2 in \cite{Popescu}.
\begin{theorem}\label{ourthm}
Let $(X,d)$ be a complete metric space and $T:X\rightarrow X$ be a $(\varphi,F)$-contraction. We suppose that $T$ is contractive, $F\in\mathcal{R}_+$ and\\ $\varphi:(0,\infty)\rightarrow (0,\infty)$ verifies the condition $(iii')$. Then $T$ is a CJMP-contraction.
\end{theorem}
%\begin{proof}
%From hypothesis, $T$ is contractive. We prove that the condition $(b)$ in %Definition \ref{dfmc} is satisfied. If it isn't, then it follows that there is %$\varepsilon_0>0$ such that for every $\delta>0,$ there are %$x_{\delta},y_{\delta}\in X$ such that %$$\varepsilon_0<d(x_{\delta},y_{\delta})<\varepsilon_0+\delta~\textrm{and}~d(Tx%_{\delta},Ty_{\delta})>\varepsilon_0.$$
%We take $\displaystyle \delta=\frac{1}{n+1},~n\in\mathbb{N}.$ It follows that %there are two sequences $\{x_n\}_{n\in\mathbb{N}},$ $\{y_n\}_{n\in\mathbb{N}}\subset X$ such that 
%\begin{equation}\label{eqmc1}
%\varepsilon_0<d(x_n,y_n)<\varepsilon_0+\frac{1}{n+1},~n\in\mathbb{N}
%\end{equation}
%and
%\begin{equation}\label{eqmc2}
%\varepsilon_0<d(Tx_n,Ty_n),~n\in\mathbb{N}.
%\end{equation}
%Since $T$ is contractive, from (\ref{eqmc1}) and (\ref{eqmc2}), we have
%$$\varepsilon_0<d(Tx_n,Ty_n)<d(x_n,y_n)<\varepsilon_0+\frac{1}{n+1},~n\in\mathbb{N}$$
%Hence, $d(x_n,y_n)\rightarrow\varepsilon_0^+,~d(Tx_n,Ty_n)\rightarrow\varepsilon_0^+,$ as $n$
% goes to infty.\\
%\\ 
% From (\ref{eqW}), we have $$\varphi(d(x_n,y_n))+F(d(Tx_n,Ty_n))\leq F(d(x_n,y_n)),~n\in\mathbb{N}$$
% hence $$\liminf_{n\rightarrow\infty}\varphi(d(x_n,y_n))+F(\varepsilon_0+0)\leq F(\varepsilon_0+0),$$ contradiction.
 %\end{proof}
\begin{cor} Let $(X,d)$ be a complete metric space and $T:X\rightarrow X$ be a $(\varphi,F)$-contraction such that $F$ is continuous at right, $\varphi$ verifies $(iii')$ and $T$ is contractive. Then $T$ is a Picard operator.
\end{cor}

%\begin{theorem}
%Let $(X,d)$ be a complete metric space and $T:X\rightarrow X$ be a %$(\varphi,F)$-contraction so that $F$ is nondeacreasing and $\varphi$ verifies $(iii)$. Then $T$ is a Picard operator.
%\end{theorem}
%\begin{proof}
%We prove that $T$ is contractive. Indeed, from (\ref{eqW}) we have $$F(d(Tx,, Ty))<F(d(x,y))~\textrm{if}~Tx\neq Ty.$$
%It follows that $$d(Tx,Ty)<d(x,y)~\textrm{if}~Tx\neq Ty.$$
 %Contrary, if we have $$d(Tx, Ty)\geq d(x,y)~\textrm{if}~Tx\neq Ty.$$
  %it follows $$F(d(Tx, Ty))\geq F(d(x,y))~\textrm{if}~Tx\neq Ty$$ since $F$ is nondecreasing.\\
  %\\
  %It follows that $$d(Tx,Ty)<d(x,y)~\textrm{if}~x\neq y.$$
  %From hypothesis, $F$ is nondecreasing, hence $F\in\mathcal{R}_+.$ By Theorem \ref{ourthm}, it follows that $T$ is a Picard operator.
%\end{proof}
APPLICATION 3. The following result improve the principal theorem in \cite{Gubran}.
\begin{theorem}
Let $(X,d)$ be a complete metric space and $T:X\rightarrow X$ be a $(\varphi,F)$-contraction such that $F:(0,\infty)\rightarrow\mathbb{R}$ is nondecreasing and \\$\varphi:(0,\infty)\rightarrow(0,\infty)$ verifies condition
\begin{enumerate}[$(iii'')$]
    \item For every strictly decreasing sequence $\{t_n\}_{n\in\mathbb{N}}\subset (0,\infty),$
    $$\lim_{n\rightarrow\infty}\varphi(t_n)=0\Longrightarrow\lim_{n\rightarrow\infty}t_n=0.$$
\end{enumerate}
Then $T$ is a Picard operator.
\end{theorem}
\begin{proof} We prove that $(iii'')\Longrightarrow(iii').$\\
\\
Let be $t>0.$ If $(\exists)t_n\subset (t,\infty),~ t_n\rightarrow t$ and $\displaystyle \limsup_{n\rightarrow\infty}\varphi(t_n)=0$ it follows that $\displaystyle\lim_{n\rightarrow\infty}\varphi(t_n)=0.$\\

By the Monotone Subsequence Theorem, we can suppose that $(t_n)_{n\in\mathbb{N}}$ is nondecreasing. If an infinite number of $\{t_n\}_{n\in\mathbb{N}}$ are equal terms between them, then the sequence of these terms is a constant sequence with limit $t$, contrar to $t_n>t,$ $n\in\mathbb{N}.$

It follows that only a finite number of $\{t_n\}$ are equal terms between them. Then, there is $n_0\in\mathbb{N}$ such that for $n\geq n_0,$ $(t_n)_{n\geq n_0}$ is strictly increasing and, by hypothesis, $t=0,$ contradiction.

%Since $t_n$ is strictly decreasing $(\exists) \lim t_n:=t.$ If $t>0,$ it follows that
\end{proof}
APPLICATION 4. In Theorem \ref{ourthmlmg}, we take $E(t)=\alpha F(t),$ $t\in(0,\infty),$ where $\alpha$ is a constant in $[0,1)$ and $F\in\mathcal{R}_+.$

\begin{theorem} Let $(X,d)$ be a complete metric space and $T:X\rightarrow X$ be a mapping such that:
$$Tx\neq Ty\Longrightarrow F(d(Tx,Ty))\leq \alpha F(d(x,y)).$$
We suppose that: 
\begin{enumerate}[$(C'_1)$]
    \item for $t,s\in (0,\infty),$ $t\leq s$ $\Longrightarrow \alpha F(t)<F(s);$
    \item  $F(t+0)>0,~t\in(0,\infty)$
\end{enumerate}
Then $T$ is a Picard operator.
\end{theorem}
\begin{cor} Let $(X,d)$ be a complete metric space and $T:X\rightarrow X$ be a mapping and $F$ be such that
\begin{enumerate}
    \item $F$ is nondecreasing;
    \item $F(s)>0$, $s>0$;
    \item $Tx\neq Ty\Longrightarrow F(d(Tx,Ty))\leq\alpha F(d(x,y))$
\end{enumerate}
Then $T$ is a Picard operator.
\end{cor}

The natural question that arise is the following: Are there any non monotone functions which satisties the condition $(C'_1)$?\\
\\
The answer is affirmative. Indeed, we can take the following function:
$$F(t)=\begin{cases}\dfrac{5}{2},~0<t<\dfrac{1}{4}\\
\\
\dfrac{1+t}{\sqrt{t}},~t\geq\dfrac{1}{4}
\end{cases}$$
$F$ is continuous and verifies $$\dfrac{2}{5}F(t)<F(s),$$
for all $t,s\in(0,\infty), t\leq s.$\\
\\
APPLICATION 5. After the above results were obtain, we saw the paper \cite{Proinov2}, where the author obtain new fixed point theorems, that extend and unify many earlier results, by assuming that T satisfies
a contractive-type condition. The following result improve the Theorem 3.6 of the paper \cite{Proinov2} of P.D. Proinov.
\begin{theorem}
Let $(X,d)$ be a complete metric space and $T:X\rightarrow X$ be a mapping satisfying condition (\ref{(E,F)-contractions}), where the functions $F, E:(0,\infty)\rightarrow\mathbb{R}$ verifies the following conditions:
\begin{enumerate}[$(p_1)$]
    \item $F$ is nondecreasing;
    \item $E(t)<F(t)$ for any $t>0$;
    \item $\displaystyle \liminf_{t\rightarrow\varepsilon^+}E(t_n)<F(\varepsilon^+)$ for any $\varepsilon>0,$ if $(t_n)\subset(\varepsilon,\infty),$ $t_n\rightarrow\varepsilon.$ 
\end{enumerate}
\end{theorem}
Then $T$ is a $CJMP$-contraction, hence a Picard operator.
\begin{proof}
We apply Theorem \ref{ourthmlmg}. The condition $(C_1)$ folllows from $(p_1)$ and $(p_2)$:
$$\textrm{For}~t,s\in(0,\infty),~t\leq s\Longrightarrow E(t)<F(t)\leq F(s).$$
The condition $(C_2)$ folllows from $(p_3)$ and Lemma \ref{ourlem}:

Given $\varepsilon>0$ and $(t_n)_{n\in\mathbb{N}}\subset(\varepsilon,\infty)$ such that $\displaystyle \lim_{n\rightarrow\infty}t_n=\varepsilon,$ then for any sequence $(s_n)_{n\in\mathbb{N}},$ $\varepsilon<s_n<t_n,~n\in\mathbb{N},$ we have $$\limsup_{n\rightarrow\infty}(F(s_n)-E(t_n))>0,$$ which is equivalent with $$F(\varepsilon^+)-\liminf_{n\rightarrow\infty}F(t_n)>0$$
\end{proof}

%%% ENTER REFERENCES IN THE FORM

\begin{flushright}
 P.  G\u avru\c ta, L. Manolescu\\
\textit{
{ \normalsize Department of Mathematics, Politehnica University of Timi\c{s}oara, }\\
{ \normalsize Pia\c{t}a Victoriei no.2, 300006 Timi\c{s}oara, Rom\^{a}nia
}}\\

\hspace{5mm}{ \normalsize \textbf{E-mail}: pgavruta@gmail.com\\ laura.manolescu@upt.ro}
\end{flushright}

\end{document}